\documentclass[a4paper,11pt,reqno]{amsart}

\usepackage{amsmath}
\usepackage{amsfonts}
\usepackage{amssymb}
\usepackage{amsthm}
\usepackage[shortlabels]{enumitem}
\usepackage{graphicx}
\usepackage{color}
\usepackage{dsfont}
\usepackage{MnSymbol}
\usepackage{enumitem}
\usepackage{extpfeil}
\usepackage{mathtools}
\usepackage{url}
\usepackage[margin=1in]{geometry}
\usepackage{fancyhdr}
\usepackage[all,cmtip]{xy}

\newtheorem{thmA}{Theorem}

\newtheorem{theorem}{Theorem}[section]
\newtheorem{lemma}[theorem]{Lemma}
\newtheorem{prop}[theorem]{Proposition}
\newtheorem{proposition}[theorem]{Proposition}
\newtheorem{cor}[theorem]{Corollary}

\theoremstyle{remark}

\newtheorem{remarks}[theorem]{Remarks}

\newtheorem{question}[theorem]{Question}

\theoremstyle{definition}

\def\X{\mathfrak{X}}
\def\<{\langle}
\def\>{\rangle}
\def\-{\overline}
\def\sq{\small{\square}}
\def\ssm{\smallsetminus}
\def\G{\Gamma}

\def\g{\gamma}
\def\a{\alpha}
\def\b{\beta}
\def\e{\varepsilon}
\def\imm{\rm{im}}

\def\cd{{\rm{cd}}}

\numberwithin{equation}{section}
\newcommand{\FP}{{\rm{FP}}}

\newcommand{\R}{\mathbb{R}}

\newcommand{\Z}{\mathbb{Z}}

\newcommand{\ns}{\triangleleft}

\newcommand{\A}{\mathcal{A}}
\newcommand{\M}{\mathcal{M}}

\begin{document}

\title[Weak commutativity and finiteness properties]{Weak commutativity 
and finiteness properties of groups}

\author{Martin R. Bridson}

\address{Mathematical Institute,
University of Oxford, 
Andrew Wiles Building, ROQ,
Oxford OX2 6GG,
United Kingdom} 
\email{bridson@maths.ox.ax.uk}

\author{Dessislava H. Kochloukova}

\address
{Department of Mathematics, State University of Campinas (UNICAMP), 13083-859, Campinas, SP, Brazil} 
\email{desi@ime.unicamp.br}
\thanks{The first author was supported in part by a Wolfson Research Merit Award from the Royal Society.
The second author was supported in part by regular grant 2016/05678-3 from  FAPESP and ''bolsa de produtividade em pesquisa'' 303350/2013-0 CNPq, Brazil}


%\date{29 April 2017; 20 June 2017; Corrections Sept 2018}


\begin{abstract}  We consider the group $\X(G)$ obtained from $G\ast G$ by forcing
each element $g$ in the first free factor to commute with the copy of $g$ in the second free factor. Deceptively complicated finitely presented groups arise from this construction: $\X(G)$ is finitely presented if and only
if $G$ is finitely presented, but if $F$ is a non-abelian free group of finite rank then $\X(F)$ has a subgroup of finite index whose third homology is not finitely generated.
\end{abstract} 

\maketitle

\section*{Introduction}

In \cite{Said} Sa\"{i}d Sidki defined a functor $\X$ that assigns to a group $G$ the group
$\X(G)$  obtained from $G\ast G$ by forcing
 each element $g$ in the first free factor to commute with the copy of $g$ in the second free factor.
More precisely, taking a second copy $\-{G}$ of $G$ and fixing an isomorphism  $g\mapsto\-{g}$, one
defines $\X(G)$ to be the quotient of the free product $G\ast\-{G}$
by the normal subgroup $\<\!\langle   [g, \-g] :  g \in G \rangle\!\>$. We shall show that this construction is an intriguing new source of finitely presented groups.

At first glance, there seems little reason to believe that $\X(G)$ can be 
defined by finitely many of the given relations, even if $G$ is finitely presented, but this is in fact the case:

\begin{thmA}\label{thmA} 
$\X(G)$ is finitely presented if and only if  $G$ is finitely presented.
\end{thmA}

We shall derive an explicit finite presentation for $\X(G)$
from a finite presentation for $G$; see Theorem \ref{t:present}. For a free group of rank $2$, 
$$
\X(F_2) = \< a,b,\-a,\-b\mid [w,\-w],\ \ w\in\Upsilon_2\>,
$$
where $\Upsilon_2=\{a,\, b,\, ab,\, ab^{-1},\, a^{-1}b,\, ba,\, aba,\, a^{-1}ba,\, aba^{-1},\, bab,\, b^{-1}ab,\, bab^{-1}\}$. 
For  free groups of higher rank, the number of
commutators in our presentation of $\X(F_m)$ grows exponentially with $m$.

It would be quite wrong to interpret Theorem \ref{thmA} as saying that $\X(G)$ is a less
complicated group than one first imagines.  
Even for a group as easily understood as a free group $F$, it transpires
that $\X(F)$ is an exotic creature. For example, $\X(F)$ does not have a classifying space
with only finitely many $3$-cells. Recall that a group $G$ is of type $\FP_n$ if the trivial $\mathbb{Z} G$-module $\mathbb{Z}$ has a projective resolution that is finitely generated up to dimension $n$. For more details on groups of type $\FP_n$ the reader is refered to \cite{B-book}.

\begin{thmA}\label{thmB}
 If $F$ is a non-abelian free group, then $\X(F)$ is not of type $\FP_3$.
\end{thmA}

What we shall actually prove is that $\X(F)$ has a subgroup of finite index whose third
homology $H_3(-,\Z)$ is not finitely generated -- see Theorem \ref{t:B'}.

The study of higher finiteness properties of groups began with the work
of Serre \cite{serre} and Wall \cite{wall} in the 1960s, and it has remained
a rich and active area of research. The first example of a finitely presented
group that is not of type $\FP_3$ was constructed by Stallings \cite{stallings} in 1963.
Many groups of this sort are now known, but the functorial nature 
of the construction in Theorem \ref{thmB} is striking.

We also prove a homological version of Theorem A.
In the light of Theorem \ref{thmB}, one
cannot hope to extend this theorem beyond dimension 2.

\begin{thmA} 
 $\X(G)$ has homological type $\FP_2$ if and only if $G$ has type $\FP_2$.
\end{thmA}

In contrast to Theorem \ref{thmB},
Kochloukova and Sidki \cite{KochSidki} showed that if $G$ is a soluble group of type $\FP_{\infty}$ then $\X(G)$ is also soluble group of type $\FP_{\infty}$.  
The functor $\X$ preserves many other interesting classes of groups.
For example, Sidki
\cite[Thm.~C]{Said} showed that if $G$ lies in
any of the following classes, then $\X(G)$ lies in the same class: finite $\pi$-groups, where $\pi$ is a set of primes; finite nilpotent groups; solvable groups.
Gupta, Rocco and Sidki  \cite{G-R-S} proved that the class of finitely generated nilpotent groups is closed
under $\X$, and 
Lima and Oliveira \cite{L-O} proved the same for polycyclic-by-finite groups. 
Kochloukova and Sidki \cite{KochSidki} also
 proved a forerunner of Theorem A: if $G$ is finitely presented and $G' / G''$ is finitely generated then $\X(G)$ is finitely presented. 
 \smallskip

Our original proof of Theorem A relied on the basic structural
results for $\X(G)$ established by Sidki in his seminal paper \cite{Said}, as well
as the VSP criterion for finite presentability established
by Bridson, Howie, Miller and Short \cite[Thm.~A]{BHMS2}, and 
a finiteness result concerning the abelianisation of a special subgroup
$L(G)<\X(G)$ that was established by Lima and Oliveira \cite{L-O}. A shortcoming of this original proof
was that it did not give an estimate on the number of relations needed to present $\X(G)$; it was in seeking to address this that
we discovered the proof presented in Section 1. In this approach, a decomposition
scheme for planar diagrams leads to an inductive proof that all of the relations
$[w,\-w]$ in the definition of $\X(G)$ follow from a finite set of a particular form.
We believe that this technique is of interest in its own right and expect it
to find further applications.

In a subsequent paper \cite{BK2} we shall explore the isoperimetric properties of $\X(G)$
and prove, among other things,
 that the class of virtually nilpotent groups is closed under $\X$. 

This paper is organised as follows. In Section 1 we present our diagrammatic
proof that $\X(F)$ is finitely presented if $F$ is a finitely generated free
group; the proof yields an explicit finite presentation. Theorem A follows
easily from the special case $G=F$.
In Section 2 we recall  some of Sidki's
basic structural results for $\X(G)$ and its normal subgroups $L=L(G), D=D(G)$ and
$W=L\cap D$, and we relate Theorem A to these structures. 
In Section 3 we prove Theorem C. 
In Section 4 we
prove Theorem B. The proof
relies on an analysis of the LHS spectral sequences
associated to the maps $\X(F)\to \X(F)/W$ and $L\to L/W$; results concerning the
homology of subdirect products of free groups (\`a la \cite{BHMS1}) play a key role in this
analysis. 

\smallskip
\noindent{\em{Acknowledgement:}} We thank the anonymous referee for their comments and for generously converting our hand-drawn pictures into the diagrams in their current form.

\section{The finite presentability of $\X(G)$}

Let $F$ be the free group with basis $\{a_1,\dots,a_m,\-a_1,\dots,\-a_m\}$.
Let $\M$ be the free monoid (set of
finite words) with basis
 the $4m$-letter alphabet consisting of the symbols $a_i, \-a_i$ and 
formal inverses $a_i^{- 1},\-a_i^{- 1}$.  
Let $F(\A)$ (resp. $\M(\A)$) be the subgroup (resp. submonoid) of $F$ (resp. $\M$)
generated by $\A=\{a_1^{\pm 1},\dots,a_m^{\pm 1}\}$.
We shall
reserve the term {\em{word}} for elements of $\M$ (rather than $F$) and write $|w|$
to denote the length of a word. We continue to write $w$ when considering the
image of $w\in\M
$ in $F$ and its quotients.

Given a subset $S\subset F$, we adopt the standard terminology$\mod S$ when describing properties of the group $F/N_S$, where
$N_S$ is the normal subgroup generated by $S$. Thus, for example,
$w=1 \mod S$ means that the image of $w$ in $F/N_S$ is trivial. We write $\-w$ for the word obtained from $w\in \M(\A)$
by replacing each $a_i^{\pm 1}$ with $\-a_i^{\pm 1}$

We shall be particularly concerned with words of the form
$$\sq_w:=w^{-1}\-w^{-1}w\-w$$
for $w\in \M(\A)$. In $F$ we have $\sq_w=[w,\-w]$. Define
$$\sq(\!(n)\!)=\{[w,\-w^{-1}] \mid |w|\le n\}.$$

Theorem \ref{thmA} is equivalent to the assertion that there exists an $n_0$
such that
$\sq_u=1 \mod \sq(\!(n_0)\!)$ for all $u\in \M(\A)$. Thus our objective is to find equations in $F$ that exhibit the triviality of $\sq_w$ in $\X(F)$ as a consequence of shorter relations.

\subsection{Identities from diagrams}

The use of planar diagrams to explore relations in groups is well established. Most applications can be interpreted as forms of van Kampen's Lemma, and this applies to the following simple observation.

\begin{lemma}\label{l0} Let $X$ be a finite, connected, planar graph with edges
labelled by elements of the free group $F$. Let $r_1,\dots,r_m$ be the elements labelling
the boundary cycles of the compact connected components of $\R^2\ssm X$ and let $r_0$ be the label on the boundary cycle of the infinite component. (These labels can be read from any point on the cycle, proceeding with either orientation.) For all
$i,j\in\{r_0,\dots,r_m\}$, let 
$R_{ij}=\{r_k \mid k\neq i,j\}$. Then, 
$$
r_i= 1 \mod R_{ij} \ \iff \ r_j=1 \mod R_{ij}.
$$
\end{lemma} 

A topological explanation for this fact is the following: consider the  1-vertex graph with directed edges
labelled by the elements of a fixed generating set for $F$ and attach $m-1$
discs to this along the loops labelled $r_k\in R_{ij}$. The fundamental group
of the resulting 2-complex is $F/\<\!\<R_{ij}\>\!\>$ and the loop
labelled $r_i$ is freely homotopic to that labelled $r_j^{\pm 1}$ (with the
sign determined by the original choices of orientation on the boundary loops for $X$). 

An algebraic proof can be obtained by following the standard proof of van Kampen's lemma.
We illustrate this with the two special cases that are central to our proof of Theorem A.

\begin{figure}[ht] 
\includegraphics{bk1i.mps}\qquad\includegraphics{bk1ii.mps}
 \caption{diagram with boundary $\sq_{\b\e\a}=1$ }
 \vspace{-0.13cm}
 \label{fig1}
 \end{figure}

In this diagram $\a, \e,\b$ are arbitrary words
and all segments in the interior are labelled with the words labelling the 
segments of the boundary that are vertical or horizontal translates of them. We shall
read all labels by proceeding clockwise around boundary faces.  

We perform algebraic manipulations {\em{in the free group}}, guided by reading the outer boundary
of the decomposition shown in fig 1(ii).
The validity of these manipulations can be verified\footnote{discovering the decomposition algebraically is harder than
verifying it!} by freely reducing the righthand side of equation (\ref{e1}); the four bracketed factors 
correspond to the four  faces of the diagram. 

Our notational conventions are $[x,y]:=x^{-1}y^{-1}xy$ and $x^y:= y^{-1}xy$.

\begin{equation}\label{e1}
[\b\e\a, \-\b\-\e\-\a] = [\a,\-\a]\ ([\-\a^{-1},\e]\, [\a^{-1},\-\e]^{-1})^{\a\-\a}\
([\-\a^{-1},\,\b]^{\e\-\e^{-1}}\, [\-\b,\,\a^{-1}])^{\a\-\e\-\a} \
[(\-\b\-\e)^{-1},\, \b\e]^{\a\-\b\-\e\-\a}.
\end{equation}

\begin{figure}[ht] 
\includegraphics{bk2i.mps}\qquad\includegraphics{bk2ii.mps}
 \caption{$[\overline{u}^{-1},\,v]\, [\overline{v},\,u^{-1}] = 1 $}
 \vspace{-0.13cm}
 \label{fig2}
\end{figure}

When $\e$ is the empty word, figure 1 simplifies to figure 2 and equation (\ref{e1})
reduces to:

\begin{equation}\label{e:zig}
[\b\a, \-\b\-\a] = [\a,\-\a]\  
([\-\a^{-1},\,\b]\, [\-\b,\,\a^{-1}])^{\a\-\a} \
[\-\b^{-1},\, \b]^{\a\-\b\-\a}.
\end{equation}

\smallskip

\begin{lemma}\label{l:2.1} For all non-empty words $\a,\e,\b\in \M(\A)$ with 
$|\a| + |\e| + |\b| = n+1$,
$$
\sq_{\b\e\a} = 1  \mod \sq(\!(n)\!)  \iff [[\-\a^{-1},\,\b],\, \e\-\e^{-1}]=1 \mod \sq(\!(n)\!).
$$
\end{lemma}

\begin{proof} Equation (\ref{e:zig}) shows that for all
words $u,v\in \M(\A)$ with $|u|+|v|\le n$ we have 
\begin{equation}\label{e:star}
[\-u^{-1},\,v]\, [\-v,\,u^{-1}] = 1 \mod \sq(\!(n)\!).
\end{equation}
Thus,  mod $ \sq(\!(n)\!)$,
each factor on the
right hand side of (\ref{e1}) is trivial, except possibly the third, and for that  we
have,
$$
[\-\a^{-1},\,\b]^{\e\-\e^{-1}}\, [\-\b,\,\a^{-1}] =1
\iff 
[\-\a^{-1},\,\b]^{\e\-\e^{-1}} = [\-\b,\,\a^{-1}]^{-1} = [\-\a^{-1},\,\b],
$$
where the last equality is an application of (\ref{e:star}).
\end{proof}

\smallskip

\begin{lemma}\label{l2} For all non-empty words $\a,\e,\b\in \M(\A)$ with 
$|\a| + |\e| + |\b| = n+1$,
$$
\sq_{\b\e\a} = 1   \mod \sq(\!(n)\!) \iff \sq_{\b\e^{-1}\a} = 1   \mod \sq(\!(n)\!).
$$
\end{lemma}

\begin{proof} The second of the equivalences in the following chain is justified
because the identity ${\-\e\e^{-1}} = {\e^{-1}\-\e} [\-\e,\e^{-1}]$ implies
$(\e\-\e^{-1})^{-1} = {\e^{-1}\-\e} \mod\sq(\!(n-1)\!)$. The other equivalences are instances
of Lemma \ref{l:2.1}. All calculations are mod $ \sq(\!(n)\!)$.
$$
\sq_{\b\e\a} = 1  \iff [[\-\a^{-1},\,\b],\, \e\-\e^{-1}]=1 \iff
[[\-\a^{-1},\,\b],\, \e^{-1}\-\e]=1 \iff
\sq_{\b\e^{-1}\a}= 1.
$$
\end{proof}
 
\subsection{A reduction scheme}

We maintain the notation established at the beginning of this section; in particular $F$ is the free group of rank $m$.
The following proposition shows that in order to present $\X(F)$ one only needs to impose the relations $\sq_u$ with
$|u|\le m+1$.

\begin{prop}\label{p:G=F}
For all $w\in  \M(\A)$, if $|w|\ge m+2$ then $\sq_w=1 \mod
\sq(\!(m+1)\!)$.
\end{prop}

\begin{proof} It is enough to prove that $\sq_w=1 \mod
\sq(\!(|w|-1)\!)$ if $|w|\ge m+2$.

Consider a word $w=\b_0 w_0 \a_0$ of length $n+1 \geq m+2$ with $\a_0,w_0,\b_0\in  \M(\A)$
and $|\a_0|=|\b_0|=1$. By making repeated use of Lemma \ref{l2}, we can
permute the letters of $w_0$
$$
\ldots ab\ldots \sim \ldots (ab)^{-1}\ldots = \ldots b^{-1}a^{-1}\ldots \sim \ldots ba\ldots,
$$
and thus, after free reduction, transform $w_0$ into a word of the form
$
w_0'= a_{i_1}^{p_1}\dots a_{i_s}^{p_s}
$
with $|w_0'|\le |w_0|$ and  $1\le i_1 < i_2<\dots < i_s\le m$, such that
\begin{equation}\label{e3}
\sq_w = 1  \mod \sq(\!(n)\!) \iff \sq_{\b_0 w_0' \a_0} = 1   \mod \sq(\!(n)\!).
\end{equation}
Using Lemma \ref{l2} again, we can assume $p_i>0$.
And by transforming subwords $aa$ into $aa^{-1}$ and cancelling, we can actually
assume that $w_0'$ is a monomial word,
i.e. each $p_i=1$. When modified in this way, $|w_0'|\le m$. Next,
if $\a_0 = a_k^{\pm 1}$ for some $k\in\{i_1,\dots,i_s\}$, then we can permute
$a_k$ to the end of $w_0'$ and cancel it with $\a_0$
(using Lemma \ref{l2} to replace it by its inverse if necessary). Thus we can
assume that $k\not\in \{i_1,\dots,i_s\}$. Similarly, $\b_0 = a_j^{\pm 1}$
with $j\not\in \{i_1,\dots,i_s\}$ (but it is possible that $a_j=a_k$).
With these final assumptions, we have $|w_0'|\le m-1$, so from (\ref{e3})
we have $\sq_w=1 \mod \sq(\!(n)\!)$, 
provided $n\ge m+1$.
\end{proof}

\subsection{A presentation of $\X(F)$}

First we consider the case where $F=F_2$ is free of rank $2$ with basis $\{a,b\}$. The proof of Proposition \ref{p:G=F} shows that in order to present $\X(F_2)$ it is sufficient to
impose 12 relations of the form $\sq_w = 1$. 
To obtain the list of sufficient relations, 
one begins with all words of length at most 2 and adds to it
those of length 3 obtained by following the proof of 
Proposition \ref{p:G=F} with $w_0'=a$
and $w_0'=b$. The list can then be shortened by noting that $\sq_w=1$ implies $\sq_{w^{-1}}=1$.   
$$
\X(F_2) = \< a,b,\-a,\-b\mid [w,\-w]\ \text{ if } w\in\Upsilon_2\>,
$$
where $\Upsilon_2=\{a,\, b,\, ab,\, ab^{-1},\,a^{-1}b,\,ba,\,
aba,\, a^{-1}ba,\, aba^{-1},\, bab,\, b^{-1}ab,\, bab^{-1}\}$.  

Consider now the case where $F=F_3$ is free of rank $3$ with basis $\{a,b,c\}$.
In this case, we include 3 words
of length 1 and 4 words
of length 2 for each pair of distinct basis elements. To these 15
 words we add
further relations accounting for each of the possibilities for $w_0'$ in the
proof of Proposition \ref{p:G=F} (in its final,
positive, square-free form). The possibilities are
$$
a,\ b,\ c,\ ab,\ bc,\ ac.
$$
In each case we have to account for the different initial and terminal letters $\a_0$
and $\b_0$ (with signs). If $w_0'=a$, there are 10 possibilities (after removing
redundancies such as $bac \sim (bac)^{-1}\sim c^{-1}ab^{-1}$ by appealing to Lemma \ref{l2}):
$$
bab,\, bab^{-1},\, b^{-1}ab,\, bac,\, bac^{-1},\ b^{-1}ac,\, cab,\, cac,\, cac^{-1},\,
c^{-1}ac.
$$ Likewise we add 10 words for $w_0'=b$ and 10 for $w_0'=c$.

For $w_0'=ab$, the relations that we add are $\sq_w$ with $w \in \{cabc,\, cabc^{-1},\,  c^{-1}abc\}$. And similarly for $w_0'\in\{bc,ac\}$. 

Thus we obtain a presentation of $\X(F_3)$ with generators $\{a,b,c,\-a,\-b,\-c\}$
subject to $54$ relations of the form $[w,\-w]=1$, with $|w|\le 4$.

In the general case, where $F_m$ is free of rank $m$, say with a basis $ a_1, \ldots, a_m$, our algorithm yields $m$
words of length $1$ together with $2m(m-1)$ words of length $2$ and, for
each $s<m$ and each of the  $\binom{m}{s}$ monomials $w_0 = a_{i_1}\dots a_{i_s}$ with
$1\le i_1<\dots<i_s\le m$, an additional $3(m-s)$ relations of the form $a_k w_0 a_k, a_k w_0 a_k^{-1}, a_k^{-1} w_0 a_k$ where $k \in \{1, \ldots, m \} \setminus \{ i_1, \ldots, i_s \}$ and $4\binom{m-s}{2}$ relations of the form $a_k w_0 a_j, a_k w_0 a_j^{-1}, a_k^{-1} w_0 a_j, a_j w_0 a_k$ where $k \not=j$ and $k,j \in \{1, \ldots, m \} \setminus \{ i_1, \ldots, i_s \}$. Thus $\X(F_m)$
has a presentation with $2m$ generators and 
\begin{equation} \label{v-m}
v_m = m + 2m(m-1) + \sum_{s=1}^{m-1} \binom{m}{s} \left[3(m-s) + 4\binom{m-s}{2}\right]
\end{equation}
relations of the form $[w,\-w]=1$, with $|w|\le m+1$. Note that as $m$
increases, this number of relations grows exponentially.

A somewhat
less economical presentation is given by the following distillation of the proof of Proposition \ref{p:G=F}.

\subsection{A finite presentation of $\X(G)$}

\begin{theorem}\label{t:present}
Let $G = \<a_1,\dots,a_m \mid r_1,\dots,r_n\>$ be a finitely presented group. Then
$$
\X(G) = \< a_1,\dots,a_m, \-a_1,\dots,\-a_m \mid r_1,\dots,r_n, \-r_1,\dots,\-r_n,  [w,\-w]\ \text{ if } w\in\Upsilon_m\>,
$$
where $\Upsilon_m$ consists of the elements of
$\{a_i, a_ia_j, a_ia_j^{-1}, a_i^{-1} a_j, a_ja_i : 1\le i < j \le n\}$ together with words of the
form $w=\b_0w_0\a_0$ where
$w_0$ is a monomial word $a_{i_1}\dots a_{i_s}$ with $1\le i_1<\dots<i_s\le m$,
and $\a_0=a_j^{\pm 1}$ and $\b_0=a_k^{\pm 1}$ with $j,k\not\in\{i_1,\dots,i_s\}.$
\end{theorem}

\begin{proof} It follows immediately from the definition that $\X(G)$ 
is the quotient of $\X(F_m)$ by the normal subgroup generated by the $r_i$ and $\-r_i$,
and the presentation of $\X(F_m)$ comes from the proof of Proposition \ref{p:G=F}.
\end{proof}

\section{On the structure of the group ${\X}(G)$} \label{prelimX(G)}

Following Sidki \cite{Said}, we highlight some subgroups and decompositions of $\X(G)$.
We follow the notations from \cite{Said} except that we write $\-G$ and $\-g$
where Sidki writes $G^\psi$ and $g^\psi$.

By means of careful commutator calculations, Sidki proves
that the following subgroups of $\X(G)$ are normal \cite[Lemma~4.1.3]{Said}:
%By definition
$$
D = D(G) := [G,\-{G}];\ \ L = L(G) :=  \langle  g^{-1} \-{g} \mid g \in G  \rangle .$$
It follows that $D$ is the kernel of the natural map $\X(G)\to G\times\-{G}$ and that
$L$ is the kernel of the map $\X(G)\to G$ that sends both $g$ and $\-{g}$ to $g$. This last map has an obvious splitting, giving
$$
\X(G) = L\rtimes G.
$$
By combining the maps with kernel $D$ and $L$, we obtain a homomorphism
$$
\rho: \X(G) \to G\times \-{G} \times G \cong G\times G\times G$$
with 
$$
\rho(g) = (g,g,1), \ \ \rho(\-g) = (1,g,g) \hbox{ for all } g \in G.
$$
The kernel of $\rho$ is
$$\ W = W(G) := D \cap L.$$
By means of further commutator calculations, Sidki  \cite[Lemma~4.1.6~(ii)]{Said} showed that $D$ commutes with
$L$ and therefore $W$  is central in $DL$ (cf. Lemma \ref{l:2.1} above); in particular, $W$ is abelian. 

The image of $\rho$ is
$$
 Q_G:={\imm ( \rho)} = \{ (g_1, g_2, g_3) \mid g_1 g_2^{-1} g_3 \in [{G}, G] \}.$$
Note that $\imm ( \rho)$ contains the commutator subgroup of $G\times G \times G$;
 $\imm ( \rho)$ is normal in  $G\times G \times G$ with quotient $H_1(G, \mathbb{Z})$. Moreover, its projection onto each pair of coordinates is onto, so by the Virtual Surjection to Pairs Theorem \cite{BHMS2}, if
$G$ is finitely presented then $\imm ( \rho)$ is finitely presented. Thus, when $G$ is finitely presented we have an exact sequence 
$$
1\to W\to \X(G)\to Q_G\to 1
$$
with $Q_G$ finitely presented and $W$ abelian.

%Since $L$ is normal in $\X(G)$ and $\rho(\X(G))$ maps onto $G\times 1\times G\cong G$, we  see that 
The group $\rho(L)$ is normal in $G\times 1\times G\cong G\times G$; the intersection
with each direct factor is the commutator subgroup and the quotient is $H_1(G, \mathbb{Z})$.
If $G$ is finitely generated then $\rho(L)$ is finitely generated (see \cite{BM}, for example), but if $G=F$ is a non-abelian free group then $H_2(\rho(L),\Z)$ is not finitely generated. Thus, in the case $G=F$, we have a central extension
\begin{equation}\label{L:ses}
1\to W \to L \to \rho(L)\to 1,
\end{equation} 
where $\rho(L)$ is finitely generated but not finitely presented.
%We shall see in section \ref{s:last} that for $G=F$ neither $H_2(\rho(L),\Z)$ nor $W$ is finitely generated.

\subsection{Finite generation for $L$}

We noted above that $\X(G)= L\rtimes G$. So if $\X(G)$ is finitely generated then $G$, as a retract, is also finitely generated. 
Our initial attempts to prove Theorem A focused on 
a partial converse to this observation, the relevance of which is explained 
by the following lemma.

\begin{lemma} \label{Lcriterion} Let $F$ be a finitely generated free group.  If $\X(F)$ is finitely presented then $L = L(F)$ is finitely generated.
\end{lemma}
\begin{proof} Let $\pi : \X(F)=L\rtimes F \to F$ be the retraction with kernel $L$
and recall that $[L,D] = 1$.
For a fixed $g_0 \in D \setminus L$ consider $\pi(g_0) \in F$. Marshall Hall's theorem \cite{Hall} provides a subgroup  $F_0<F$ of finite index such that $\pi(g_0)$ is a primitive element of $F_0$. Consider $G_0 = \pi^{-1} (F_0) = L \rtimes F_0$.
As a subgroup of finite index in $\X(F)$, this is finitely presented. By choosing a 
different section of $\pi$ if necessary, we may assume that $g_0\in \tilde F_0 = 1\times F_0$.
Since $\pi(g_0)$ is primitive in $F_0$, we know that $g_0$ is primitive in 
$\tilde F_0$, so $\tilde F_0 = K * F_1$ where $K$ is the cyclic subgroup of $\tilde F_0$ generated by $g_0$. Thus $G_0$ decomposes as an HNN extension with stable
letter $g_0$ and relative
presentation
$$
G_0 = \langle L \rtimes F_1, g_0 \mid [L, g_0] = 1 \rangle.
$$
Crucially, $G_0$ retracts
onto $L \rtimes F_1$ by the map that sends $g_0$ to $1$. Hence $L \rtimes F_1$ is finitely presented.
By a result of Miller \cite[Lemma~2.1]{Miller}, if an HNN extension is finitely presented and the base group is finitely presented, then the associated subgroup must be finitely generated. Thus $L$ is finitely  generated. 
\end{proof}

With Theorem A in hand, this lemma tells us that $L$ is finitely generated,
but one can prove this more directly using the following lemma.

\begin{lemma}\label{l:ses} Let $\G$ be a group and suppose that $C<\G$ is central. 
If $\G/C$ and $H_1(\G,\Z)$ are finitely generated, then $\G$ is finitely
generated.
\end{lemma}

\begin{proof} Let $S\subset \G$ be a finite set whose image generates $\G/C$
and let $\Sigma<\G$ be the subgroup generated by $S$. Each $\g\in\Gamma$
can be written in the form $\gamma=z\sigma$ with $z\in C$ and $\sigma\in\Sigma$.
For all $x\in\Sigma$ we have $x^\gamma = x^\sigma\in\Sigma$. Thus $\Sigma$
is normal in $\G$. The central subgroup $C<\G$ maps onto $\G/\Sigma$, which is
therefore abelian. So if the images of $c_1,\dots,c_n\in\G$
generate $H_1(\G,\Z)$, then $S\cup\{c_1,\dots,c_n\}$ generates $\G$.
\end{proof}

We can apply this lemma in the setting of (\ref{L:ses}) because Lima and
Oliveira \cite{L-O} proved that $H_1(L,\Z)$ is finitely generated
whenever $G$ is finitely generated. We recall their
construction. 

They consider the wreath product $H=\Z\wr G$, viewing $\oplus_{g\in G}\Z$
as the additive group of the integral group ring $\Z G$; so $H=\Z G\rtimes G$.
Focusing on the
augmentation ideal $A(G)$ of $\mathbb{Z} G$ and the additive group $I_2(G)$ of the ideal 
of $\Z G$ generated by
$\{ (1-g)^2 \mid g\in G\}$, they prove that the homomorphism 
of groups $\nu:\X(G)\to H/I_2(G)$ defined\footnote{note that
$(0,1)\in H$ is the identity element, not $(1,1)$} by $\nu(g) = (1,g)$ and $\nu(\-g)
= (1,g)^{(1,1)}$ restricts to an epimorphism from $L = L(G)$ to
$A(G)/I_2(G)$ with kernel $[L,L]$. Thus the abelianisation of $L$ is a subgroup
of $M=\Z G/I_2(G)$. Lima and
Oliveira show in \cite{L-O} that if $G$ is generated by $a_1,\dots,a_m$,
then $M$ is generated as an abelian group by $1$ and the finitely many
monomials $a_{i_1}\dots a_{i_s}$ with $1\le i_1<\dots <  i_j\le m$.  

{\em{Thus we see the monomials that played a crucial role in our proof of Theorem A appear from an alternative perspective.}}

Moving beyond the case $G=F$, we record:

\begin{proposition}\label{p:Lfg} For all finitely generated groups $G$, the 
group $L(G)$ is finitely generated.
\end{proposition}

\begin{proof} The natural map
$\X(F)\to \X(G)$ sends $L(F)$ onto $L(G)$ and therefore the latter is finitely generated.
\end{proof}

By following the constructions given above, one can give an explicit generating set for $L(G)$. We illustrate this in the 2-generator case, for which we need the following lemma.

\begin{lemma}\label{l:fib}
If $G$ is generated by $\{a,b\}$, then $P= \langle (g,g^{-1}) : g\in G \rangle 
<G\times G$ is generated by $\{(a,a^{-1}),\, (b,b^{-1}),\, (ab, (ab)^{-1})\}$.
\end{lemma}

\begin{proof} $P$ is the kernel of the map $G\times G\to H_1(G,\Z)$
that restricts to the canonical surjection on each factor. It is clear that this
is generated by $\{\alpha=(a,a^{-1}),\, \beta=(b,b^{-1})\}$ and the
normal closure of $c=([a,b],1)$ (with $(1,[b^{-1},a^{-1}])$ obtained as $c [\alpha, \beta]^{-1}$), and one can obtain any conjugate of $c$ by noting that $c^{(a,1)}=c^\a$
and $c^{(b,1)}=c^\b$. Thus $P$ is generated by $\{\a,\b,c\}$. And since
$$
([a,b],1)= (a^{-1},a)\, (b^{-1},b)\, (ab, (ab)^{-1}),
$$
$P$ is also generated by $\{(a,a^{-1}),\, (b,b^{-1}),\, (ab, (ab)^{-1})\}$.
\end{proof}

\begin{proposition} If $F$ is the free group with basis $\{a,b\}$,
then $L(F)<\X(F)$ is 
generated by $\{a\-a^{-1},\, b\-b^{-1},\, ab(\-a\-b)^{-1}\}$.
\end{proposition}

\begin{proof} Applying the proof of Lemma \ref{l:ses} to (\ref{L:ses}), we see that
it is enough to prove that the image of the given subset of $L$ generates both $\rho(L)$
and $H_1(L,\Z)$. For $\rho(L)$, this is covered by Lemma \ref{l:fib}.

Lima and Oliveira show in \cite{L-O} that $M=\Z F/I_2(F)$ is generated as an abelian group by $X=\{1,a,b,ab\}$.
The augmentation map $\Z F\to\Z$ (defined by $w\mapsto 1$ for all $w\in F$) descends to
an epimorphism $\mu:M\to\Z$ defined by $\mu(x)=1$ (the generator of $\Z$) for each $x\in X$. Hence $\ker \mu = A(F)/I_2(F)$ is generated as an abelian group by $\{1-a, 1-b, 1 - ab \}$.
And $\nu(u\-u^{-1})= 1-u$ defines
an epimorphism $L\to A(F)/I_2(F)$ with kernel $[L,L]$.
\end{proof}

\begin{cor} $H_1(L(F_2),\Z)\cong\Z^3$.
\end{cor}

\begin{proof} $L$ is 3-generator and maps onto $H_1(\rho(L),\Z)$, which a
calculation with the 5-term exact sequence shows to be $\Z^3$; see
\cite[Theorem~A]{BM}.
\end{proof}

\section{Theorem B: Sidki's functor $\X$ preserves $\FP_2$}

We remind the reader that a group $G$ is of type $\FP_2$ if there is a 
projective resolution of the trivial $\Z G$-module $\Z$ that is finitely
generated up to dimension $2$. If
one expresses $G$ as the quotient of a finitely generated free group $G=F/R$
and considers $M=R/[R,R]$ as a $\Z F$-module (or $\Z G$-module)
via the action of $F$ on $R$ by conjugation, then $G$ is of type $\FP_2$ if
and only if $M$ is finitely generated as a module. 

For completeness, we prove the following well known fact; the forward direction is  \cite[Lemma~2.1]{BS}.

\begin{lemma}\label{l:FP2} A group $G$ is of type $\FP_2$ if and only if  there
is a finitely presented group $H$ and an epimorphism $H\to G$ with perfect
kernel.
\end{lemma}

\begin{proof} First assume that $G$ is of type $\FP_2$ and write  $G=F/R$ where $F$ is free of finite rank 
and $R/[R,R]$ is generated as a $\Z F$ module by the image of a finite set 
$S = \{r_1,\dots,r_m\}\subset R$. In $F$, every element of $R$ can be expressed as a product 
$$
r = \rho \prod_{i=1}^n \theta_ir_{j(i)}^{\pm 1}\theta_i^{-1},
$$
where $\rho\in [R,R]$ and $\theta_i\in F$.

Let $N\ns F$ be normal closure of $S$ in $F$, let $H=F/N$
and consider the natural surjection $\pi:H\to G$.
For each $rN\in R/N = \ker\pi$, the equation above projects to an equality $rN=\rho N$
in $H$, and hence $\ker \pi = R/N = [R,R]/N$ is perfect. 

Conversely, if $H=F/N$ is finitely presented, with $N=\<\!\<r_1,\dots,r_m\>\!\>$ say,
and $H\to G$ is a surjection with perfect kernel $R/N$, then by hypothesis every
$r\in R$ lies in $[R,R]N$. Thus we can express $r$ in the free group as a product of
the above form. This shows that $M=R/[R,R]$ is generated as a $\Z F$ module
by the images of $\{r_1,\dots,r_m\}$. And as $M$ is the relation
module for $G=F/R$, this implies that $G$ is of type $\FP_2$.
\end{proof}

\begin{theorem} $G$ is $\FP_2$ if and only if $\X(G)$ is $\FP_2$.
\end{theorem}

\begin{proof} 
Every retract of a group of type $\FP_m$ is $\FP_m$, so if $\X(G)=L\rtimes G$ is $\FP_m$
(for any $m$) then so is $G$.

Conversely, if $G$ is $\FP_2$ then by the lemma there is 
an epimorphism $\pi:H\to G$ with kernel $K$ where $H$ is finitely
presented and $K$ is perfect. The product of perfect subgroups is perfect and
the normal closure of a perfect subgroup is perfect, so the normal subgroup $N<\X(H)$ 
generated by $K\cup \-{K} \subset H\cup \-H$ is perfect. $N$ is the
kernel of the map $\pi : \X(H) \to \X(G)$ defined by $h\mapsto \pi(h)$ and $\-h\mapsto\-{\pi(h)}$. From Theorem A we know that $\X(H)$ is finitely presented, so $\X(G)$ is
of type $\FP_2$, by Lemma \ref{l:FP2}.
\end{proof} 

 In \cite{B-B} Bestvina and Brady constructed groups that  are of type $\FP_{\infty}$ but are not finitely presented. Recently, Leary \cite{L} showed that there are uncountably many groups with these properties.

\begin{cor} If $G$ is a group of type $\FP_2$ that is not finitely presented,
then $\X(G)$ is a group of type $\FP_2$ that is not finitely presented.
\end{cor}

\section{$\X(F)$ is not of type $\FP_3$}

We prove the following strengthening of Theorem B.

\begin{theorem}\label{t:B'} If $F$ is a free group of finite rank $m\ge 2$, then there is a subgroup of finite index
$\G<\X(F)$ with $H_3(\G,\Z)$ not finitely generated.
\end{theorem}

We assume that the reader is familiar with the Lyndon-Hochschild-Serre (LHS)
spectral sequence associated to a short exact sequence of groups.

\begin{lemma}\label{last1} Let  $F$ be a free group of finite rank $m\ge 2$. If $H_3(\G,\Z)$ is finitely generated for all subgroups 
of finite index $\G\le\X(F)$, then
$W(F)$ is not finitely generated.
\end{lemma}

\begin{proof} $W=W(F)$ is the kernel of the natural map $\rho:\X(F)\to F\times F\times F$
described in Section 2, the image of which is a finitely
presented, full subdirect product
of infinite index. Every such subgroup of $F\times F\times F$ 
is finitely presented but has a subgroup of finite
index $\Delta$ with $H_3(\Delta,\Z)$ not finitely generated \cite[Thm.~C]{BHMS1}. If
the abelian group $W$ were finitely generated, then it would
be of type $\FP_\infty$, and from the LHS spectral sequence 
$E^2_{p,q} = H_p(\Delta, H_q(W, \mathbb{Z}))$ that converges to $H_{p+q}(\rho^{-1}(\Delta), \mathbb{Z})$,
we would conclude
that $H_3(\rho^{-1}(\Delta),\Z)$ 
was not finitely generated, contrary to hypothesis. 

In more detail, $E^{\infty}_{3,0}$
is obtained from $E^2_{3,0}=H_3(\Delta,\Z)$ by passing to the kernel of 
$d_2: E^2_{3,0}\to E^2_{1,1}=H_1(\Delta, W)$ and then the kernel of $d_3 : E_{3,0}^3 \to E_{0,2}^3$,
whose codomain is a quotient of $H_2(W, \Z)$; if $W$ were finitely generated then $H_1(\Delta, W)$ would be as well (since $\Delta$ is finitely genereated), as would $H_2(W, \Z)$; but this would imply
that $E^{\infty}_{3,0}$ (whence $H_3(\rho^{-1}(\Delta),\Z)$)
 was not finitely generated. 
\end{proof}

\begin{lemma}\label{last2} Let $1\to A\to G\to Q\to 1$ be a central extension 
of a finitely generated group $Q$ of cohomological dimension $\cd(Q)\le 2$.
If the abelianisation of $Q$ is infinite
and $A$ is not finitely generated, then $H_2(G,\Z)$ is not finitely generated.
\end{lemma}

\begin{proof} As $\cd(Q)\le 2$, all of the non-zero terms on the
$E^2$ page of the LHS spectral sequence in homology associated to the given extension
are concentrated in the first three columns. Thus $E^3=E^\infty$ and the $d_2$
differentials beginning and ending at $E^2_{1,1}=H_1(Q, A)$ are
both zero maps. So $E^\infty_{1,1}$, which is a section of $H_2(G,\Z)$,
is  $H_1(Q,A) \cong H_1(Q,\Z)\otimes A$,
which contains a copy of $A$, since $H_1(Q,\Z)$ has $\Z$ as a direct summand.
As $A$ is not finitely generated, neither is $H_2(G,\Z)$.
\end{proof}

We shall need the following homological variation 
of Miller's lemma \cite[Lemma~ 2]{Miller}.
(A more comprehensive generalisation is described in Remark  \ref{remark} (2) below.)

\begin{lemma}\label{last3} If $H_n(B,\Z)$ is finitely generated but $H_{n}(C,\Z)$
is not, then $H_{n+1}(\G,\Z)$ is not finitely generated
for any HNN extension of the form $\G=B\ast_C$. 
\end{lemma}

\begin{proof} The Mayer-Vietoris sequence for the HNN extension contains an exact
sequence
$$  H_{n+1}(\G,\Z) \to H_{n}(C,\Z) \to H_{n}(B,\Z).
$$ 
\end{proof}

\noindent{\bf{Proof of Theorem \ref{t:B'}}} We shall derive a contradiction from the
assumption that $H_3(\G,\Z)$ is finitely generated for all subgroups 
of finite index $\G<\X(F)$.

As in the proof of Lemma \ref{Lcriterion},
there is a subgroup of finite index $F_0<F$ such that
$\G:=L \rtimes F_0$, which has finite index in $\X(F)$,
is an HNN extension $B{\ast}_L$
with base group $B:=L \rtimes F_1$ and associated subgroup $L = L(F)$; the
stable letter commutes with $L$.

$\G$ is finitely presented and $H_3(\G,\Z)$ is finitely
generated. $B$, as a retract of  $\G$, inherits these properties. By Lemma \ref{last3}, this
implies that $H_2(L,\Z)$ is finitely generated. By applying
Lemma \ref{last2} to the short exact sequence $1\to W\to L\to \rho(L)\to 1$,
noting that $\rho(L)<F\times F$ has cohomological dimension 2,  we deduce that
$W$ is finitely generated. This contradicts Lemma \ref{last1}, so the proof is
complete. \qed

\begin{remarks} \label{remark}  (1) Care is needed in interpreting the
logic of the above proof: we did not prove that $W(F)$ and $H_2(L(F), \mathbb{Z})$
are not finitely generated.

\smallskip

(2)  The more comprehensive homological version 
of Miller's lemma alluded to above is the following:
Suppose $G$ is the fundamental group of a
finite graph of groups where all vertex groups $G_v$ are of type $\FP_m$;
if $G$ is of type $\FP_m$, then all of edge groups $G_e$ are of type  $\FP_{m-1}$.
\end{remarks}

\subsection{Related observations}

\begin{proposition}\label{Wfg} Let $G$ be a finitely presented group. Then $W(G)$ is finitely generated if $H_2(P,\Z)$ is finitely generated, where
$P =\langle (g,g^{-1}) : g\in G \rangle$ is the kernel of the map $G\times G\to G/G'$
that restricts to the natural surjection on each factor.
\end{proposition}

\begin{proof} Consider the LHS spectral sequence in homology  for the central extension  $1\to W\to L \to \rho(L)\to 1$
and note that $\rho(L)=P$, so
the spectral sequence $
E^2_{p,q} = H_p(P, H_q(W, \mathbb{Z}))$ 
 converges to $H_{p+q}(L, \mathbb{Z})$.
Then $E^{\infty}_{1,0} = E^2_{1,0} = H_1(P, \mathbb{Z})$, $E^{\infty}_{0,1} = E^3_{0,1} = {\rm{Coker}} (d^2_{2,0})$, where $d^2_{2,0} : E^2_{2,0} = H_2(P, \mathbb{Z}) \to E^2_{0,1}$, and there is an exact sequence
$$
E^{\infty}_{0,1} \to H_1(L, \mathbb{Z}) \to E^{\infty}_{1,0}
$$
By a result of Lima and Oliveira \cite{L-O} (or Proposition \ref{p:Lfg} above)
$H_1(L, \mathbb{Z})$ is finitely generated.
Hence $E^{\infty}_{0,1} = {\rm{Coker}}(d^2_{2,0})$ is finitely generated, and
since $H_2(P, \mathbb{Z})$ is finitely generated this implies $E^2_{0,1} = H_0(P, W) = W$ is finitely generated.
\end{proof}

\begin{cor} When $H_2(P,\Z)$ is finitely generated,  $L$ is finitely presented
(respectively, ${\rm{FP}}_k$) if and
only if $P$ is finitely presented (respectively, ${\rm{FP}}_k$).
\end{cor}
\begin{proof} Proposition \ref{Wfg} gives a short exact sequence $1 \to W \to L  \to P \to 1$ and $W$ is a finitely generated abelian group.
\end{proof}

From the 1-2-3 Theorem  \cite{BBMS}, \cite{BHMS2} we get the following consequence of 
Proposition \ref{Wfg}:

\begin{cor} \label{Cor} If $G$ is finitely presented and $[G,G]$ is finitely generated, then $W(G)$ 
is finitely generated and $L(G)$ is finitely presented.
\end{cor}

\begin{proof} Under these hypotheses, the 1-2-3 Theorem  implies that 
$P=\rho(L)$ is finitely presented, hence $W$ is a finitely generated abelian group
and $L$, which fits into the exact sequence $1\to W\to L \to P\to 1$, is finitely
presented (and is of type $\FP_k$ if and only if $P$ is).
\end{proof}

\begin{remarks} 
(1) In \cite{KochSidki}  Kochloukova and Sidki showed that when the abelianization of $[G,G]$ is finitely generated then $W(G)$ is finitely generated.

(2) The converse of Proposition \ref{Wfg} is not true. For example,
if $G  = \< a,t\mid t^{-1}at=a^2\>$ then $H_2(P,\Z)$ is not
finitely generated (\cite{BBHM}, Example 2) but Kochloukova and Sidki \cite{KochSidki}
prove that $W(G)=0$. 
\end{remarks}

\section{More on the structure of $W(G)$}
Let $v_m$, as in (\ref{v-m}), be the number of relators $[w,\-w]$ required
to present $\X(F_m)$.

\begin{lemma}\label{l1} If $G$ is $m$-generator, i.e. $d(G)=m$, then $d(H_2(\X(G), \mathbb{Z}))
\le 2\, d(H_2(G, \mathbb{Z})) + \upsilon_m$.
\end{lemma}

\begin{proof} Let $N$ be the kernel of the natural map $G\ast\-G\to\X(G)$; it is the
normal closure of $\upsilon_m$ commutators and
therefore $H_0(\X(G), H_1(N, \mathbb{Z})) $ requires at most $\upsilon_m$ generators.
From the 5-term exact sequence in homology associated to  $G\ast\-G\to\X(G)$ we get
an exact sequence
$$
H_2(G\ast\-G, \mathbb{Z}) \to H_2(\X(G), \mathbb{Z})\to H_0(\X(G), H_1(N, \mathbb{Z})) \to H_1(G\ast\-G, \mathbb{Z}).
$$
And $H_2(G\ast\-G, \mathbb{Z})  \cong H_2(G, \mathbb{Z}) \oplus H_2(G, \mathbb{Z})$.
\end{proof}

\begin{proposition} Let $G$ be a perfect $m$-generator group. If
$d(H_2(G),\Z))>\upsilon_m$, then $W(G)$ is non-trivial.
\end{proposition}

\begin{proof} As $G$ is perfect, $\rho(G) = G\times G\times G$ and $H_2(\rho( G),\Z)
\cong H_2(G,\Z)\times H_2(G,\Z)\times H_2(G,\Z)$.
From the 5-term exact sequence for $1\to W\to \X(G)\to \rho(G)\to 1$ we get an
exact sequence
$$
H_2(\X(G),\Z) \to H_2(\rho(G),\Z)\to H_0(\X(G), W),
$$
so if $H_0(\X(G),W)=0$ then $d(H_2(\X(G),\Z))\ge d(H_2(\rho(G),\Z)) =
3 \, d(H_2(G),\Z)$,
contradicting Lemma \ref{l1}.
\end{proof}

\begin{question} For $F$ a finitely generated
free group, what is $W(F)$? Does $\X(F)$ have finite cohomological dimension?
\end{question}

We shall resolve this question in a future paper \cite{BK3}. In particular we prove
that if the rank of $F$ is at least $3$, then $W(F)$ is not finitely generated.

\end{document}